\documentclass{imsart}

\usepackage{amsthm,amsmath}
\usepackage[dvipsnames,usenames]{color}
\usepackage[normalem]{ulem}
\RequirePackage[numbers]{natbib}
\RequirePackage[colorlinks,citecolor=blue,urlcolor=blue]{hyperref}
\usepackage{etex}
\usepackage{amsfonts}
\usepackage{float}
\usepackage{morefloats}
\usepackage{caption}
\usepackage{subcaption}
\usepackage{tikz}
\usepackage{textcomp}
\usepackage{verbatim}
\usepackage{pifont}
\usepackage{color}
\usepackage{lscape}
\usepackage{indentfirst}
\usepackage[normalem]{ulem}
\usepackage{booktabs}
\usepackage{latexsym}
\usepackage{url}
\usepackage{epsfig}
\usepackage{graphicx}
\usepackage{amssymb}
\usepackage{bm}
\usepackage{array}
\usepackage[version=4]{mhchem}
\usepackage{ifthen}
\usepackage{multirow}
\usepackage{amstext}
\usepackage{mleftright}
\usepackage{mathtools}
\usepackage{microtype}
\usepackage{cleveref}
\usepackage{thmtools,thm-restate}
\allowdisplaybreaks

\DeclareMathOperator\erf{erf}

\newlength{\CapLen}
\AtBeginDocument{\settoheight{\CapLen}{z}}

\makeatletter
\def\@ythm#1#2#3[#4]{\def\@currentlabelname{#4}%
  \expandafter\global\expandafter\def\csname#1name\endcsname{#4}%
  \@opargbegintheorem{#3}{\csname the#2\endcsname}{#4}%
  \ifx\thm@starredenv\@undefined
    \thm@thmcaption{#1}{{#3}{\csname the#2\endcsname}{#4}}\fi
  \ignorespaces}

\makeatother

\theoremstyle{plain}
\newtheorem{theorem}{Theorem}[section]

\newtheorem{corollary}[theorem]{Corollary}

\theoremstyle{definition}
\newtheorem{definition}[theorem]{Definition}
\newtheorem{example}[theorem]{Example}
\newtheorem{prop}[theorem]{Proposition}
\newtheorem{assump}{Assumption}

\theoremstyle{remark}
\newtheorem{remark}{Remark}

\arxiv{}

\startlocaldefs
\endlocaldefs

\begin{document}

\begin{frontmatter}

\title{A generalized AIC for models with singularities and boundaries}
\runtitle{Generalized AIC}
 \author{\fnms{Jonathan D.} \snm{Mitchell}
\corref{}\ead[label=e1]{Jonathan.Mitchell@utas.edu.au}}
\address{Department of Mathematics and Statistics\\ University of Alaska Fairbanks\\ Fairbanks, AK, 99775, USA\\
}
\address{School of Natural Sciences (Mathematics)\\ University of Tasmania\\ Hobart, TAS 7001, AUSTRALIA\\
}
\address{ARC Centre of Excellence for Plant Success in Nature and Agriculture\\ University of Tasmania\\ Hobart, TAS 7001, AUSTRALIA\\
\printead{e1}
}
\author{\fnms{Elizabeth S.} \snm{Allman},}
\and
\author{\fnms{John A.} \snm{Rhodes}
\corref{}\ead[label=e2]{e.allman@alaska.edu; j.rhodes@alaska.edu}}
\address{Department of Mathematics and Statistics\\ University of Alaska Fairbanks\\ Fairbanks, AK, 99775, USA\\
\printead{e2}
}


\runauthor{J. D. Mitchell, et al.}

\begin{abstract}
The Akaike information criterion (AIC) is a common tool for model selection. 
It is frequently used in violation of regularity conditions at
parameter space singularities and boundaries.  The expected AIC is generally not asymptotically equivalent
to its target at singularities and
boundaries, and convergence to the target at nearby parameter points may be slow. We develop a generalized AIC for candidate models with or without singularities and boundaries. We show
that the expectation of this generalized form converges everywhere in
the parameter space, and its convergence can be faster than that of
the AIC. We illustrate the generalized AIC on example
models from phylogenomics, showing that it can 
outperform the AIC and gives rise to an interpolated effective number
of model parameters, which can differ substantially from the number of parameters near singularities and boundaries. We outline methods for estimating the often unknown generating parameter and bias correction term of the generalized AIC.
\end{abstract}





\end{frontmatter}

\section{Introduction}
\label{intro}

Information criteria, such as the AIC
\citep{akaike1974new}, are often used to select a candidate model from
several, or for weighting such models
\citep{burnham2004multimodel,burnham2002model}. Standard derivations
of these information criteria require regularity conditions, which
can be violated when the generating parameter is a singularity or
boundary of the parameter space. At such points, application of
standard information criteria may not be theoretically justified and
can perform poorly. Although violation of such conditions is not
uncommon in practice, such as with latent class analysis
\citep{ALCA2002,DziakEtAl2019}, or more general mixture models or
models with constrained parameters \citep{susko2020use}, robustness 
to model violations is seldom raised. In particular, the AIC is often 
used without consideration of the effect of parameter space geometry on its behaviour.

\smallskip

Given a candidate model satisfying regularity conditions with
parameter space $\Theta_0$,  
the AIC for  the sum $\mathcal{Z}_n$ of $n$ i.i.d. random observations $Z_i$
and maximum likelihood estimate $\widehat
\theta_n\in \Theta_0$ is
\begin{align*}
	\text{AIC}=-2\log{}L\left(\widehat{\theta}_n|\mathcal{Z}_n\right)+2\dim\left(\Theta_0\right).
\end{align*}
The AIC's \emph{bias correction term}, $2 \dim \Theta_0$, is twice the number of parameters of $\Theta_0$. However, this correction can be inaccurate or ambiguous when regularity conditions are violated, such as for models with singularities and boundaries, where
the notion of dimension is more nuanced.
Away from singularities and boundaries, the expected value of the AIC at generating parameter $\theta_0$ 
is asymptotically equivalent to a quantity $\Delta_{n}\left(\Theta_0,\theta_0\right)$,
called the \emph{target}. The target is related to the Kullback-Leibler divergence of the candidate model distribution at $\widehat \theta_n$ from the generating model distribution at $\theta_0$.
(See Equation~\eqref{AICtarget} for a precise definition.)
By choosing a model with minimal AIC, a practitioner seeks one that minimizes information loss.

\smallskip

When $\theta_0$ is a singularity or boundary of $\Theta_0$, however, the bias correction $2\dim\Theta_0$ is not generally asymptotically equivalent to its target. While the
likelihood term of the AIC is retained in the generalized AIC (AICg) developed here, for a parameter space with singularities or boundaries,
the AIC bias correction can be
modified to ensure the expected AICg is asymptotically equivalent to its target. Informally, if
singularities or boundaries are present, the parameter space may
``behave'' as if there is a different number of parameters than the
naive count gives. This \emph{effective number of parameters $k_e$} need not
be an integer, and can be smaller or larger than the number of
parameters $k = \dim\left(\Theta_0\right)$.

A non-integral $k_e$, such as that
appearing in the AICg bias correction, is not a new concept. It was
introduced by \citet{moody1992the} in the context of non-linear
learning systems and appears in the deviance information criterion
(DIC) of
\citet{spiegelhalter2002bayesian}. The AICg treats $k_e$ in a different way to ensure asymptotic equivalence of the expected AICg to its target.

In the AICg, $k_e$ is a function of $\theta_0$, and thus can vary across
$\Theta_0$. Since in practice $\theta_0$ is typically
unknown, methods are outlined here to estimate the AICg, the simplest of
which is to replace $\theta_0$ with $\widehat{\theta}_n$.

For a model with $k$ parameters, the expected AIC underestimates its target when
$k< k_e$, and overestimates when $k>k_e$. We derive the AICg explicitly for models from phylogenomics in Section
\ref{subsection:phyloModels}, including one where $k_e < k$ and one where $k>k_e$.

\smallskip

The definition of regular models varies considerably in the
literature, though a key part of the definition is always that Fisher
information matrices are non-singular \citep{watanabe2013waic,
  drton2009}. (See also the assumptions of Theorem 16.7 of
\citet{van2000asymptotic}). Watanabe defines singular models as those not having a one-to-one map between parameters and probability distributions and/or not having Fisher information matrices that are always positive definite. Models with singularities in their parameter spaces may not be singular according to this definition. Thus, 
while the widely applicable information criterion (WAIC)
\citep{watanabe2009algebraic} generalizes the AIC to singular
models in the sense of \citet{watanabe2013waic}, the WAIC does not address issues arising from
the geometry of the parameter space.

Complicating matters is that even when regularity conditions hold 
at parameters near singularities and boundaries,
convergence of the expected AIC may be slowed, and
a generalized form may converge faster to its target. In this sense, the AICg
can be thought of as a finite sample size correction to the
standard AIC. While the AICc \citep{sugiura1978further,
  hurvich1989regression} also has this interpretation, our generalized
form is more generally applicable.

Alternatives to the AIC include bootstrap variants, introduced by
\citet{efron1983estimating,efron1986biased}, and further reviewed by
\citet{efron1994introduction}. \citet{ishiguro1991} introduced the
WIC, while \citet{cavanaugh1997bootstrap}, \citet{shang2008bootstrap}
and \citet{seo2018information} developed variants for state-space
selection, mixed model selection and partition scheme selection,
respectively. To our knowledge, all existing AIC alternatives and adjustments assume
that $\theta_0$ is an interior point of $\Theta_0$ and not a singularity or
boundary.

The AICg accurately estimates its target, regardless of whether
$\theta_0$ is a singularity or boundary or not. It is based on
Equation~7.53 of \citet{burnham2002model}, although issues of
singularities and boundaries are not explored in that
work. Indeed, their following Equation~7.54 is not generally correct if the model
has singularities or boundaries.

Derivations of the AIC require a transformation of the space $\Theta$ that random observations $Z_i$ lie in. In the transformation $\Theta$ is scaled such that random observations in the transformed space have identity covariance. However, with covariance that does not converge to the zero matrix in the transformed space, the MLE is not generally an asymptotically unbiased estimate of the generating parameter at singularities and boundaries, a necessary condition to progress from Equation~7.53 to Equation~7.54.

While our AICg is not as simple to use as the AIC, and
thus the AIC may be preferred in standard applications, model
singularities and boundaries are common enough in complex models that
our generalized form can provide a useful improvement in many
situations. It highlights the need to consider parameter space geometry 
	when using the AIC and how this geometry might affect accuracy.

\smallskip

The example models considered in this article are from phylogenomics,
where evolutionary trees relating many species are inferred from
genomic data. Population-genetic effects, such as incomplete lineage
sorting modelled by the multispecies coalescent, result in some
inferred gene trees differing from the overall (generating) species tree.
These effects significantly complicate inference, testing, and model selection of species trees
and networks. No technical understanding of these biological processes
is required for this article, as these models are all trinomials.
Interested readers can consult \citet{mitchell2019hypothesis} and its appendices
for more biological background on our example models, and on 
hypothesis testing of models with singularities and boundaries.

\smallskip

The article is organized as follows. Section~\ref{definitions} gives
definitions, assumptions, and descriptions of example models.
The AICg is defined in Section~\ref{AICg} and the proof of the 
main theorem, that the expected AICg is asymptotically equivalent to its
target, is given. In Section~\ref{AICapplications} the AICg is derived for all
example models. Methods for estimating the bias correction, which generally depends on the generating parameter, and for estimating the generating parameter are given in Section~\ref{practicalAICg}.  To illustrate potential 
improvements of the AICg, we apply our techniques to example
models at and nearby singularities and boundaries and compare performance
to the AIC. 

\section{Definitions, assumptions and models}
\label{definitions}

\subsection{Definitions}

We define singularities and boundaries of parameter spaces as in
\citet{mitchell2019hypothesis} and \citet{drton2009}.  Let
$\mathcal{P}_{\Theta}$, with $\Theta\subseteq \mathbb R^k$, be a
parametric family of probability distributions on a measurable
space. We assume that the parameter space $\Theta$ is a semialgebraic
subset of $\mathbb R^k$; that is, it comprises points satisfying a
finite collection of multivariate polynomial equalities and
inequalities. For a semialgebraic $\Theta_0\subset \Theta$, we have a
subfamily $\mathcal{P}_{\Theta_0}$.

A \emph{singularity} of the parameter space $\Theta_0$ of
$\mathcal{P}_{\Theta_0}$ is either a) a point in $\Theta_0$ which lies
on multiple irreducible algebraic components of $\Theta_0$, or b) a
point that lies on only one component, but at which the Jacobian
matrix of the defining equations of that component has lower rank than
at generic points on the component.

Let $Cl\left(\Theta_0\right)$ denote the Zariski closure of
$\Theta_0$; that is, the points satisfying all the equalities defining
$\Theta_0$.  A subset of $\Theta_0$ is \emph{open} if it is the
intersection of $Cl\left(\Theta_0\right)$ with an open subset of
$\mathbb{R}^k$. The \emph{interior} of $\Theta_0$ is the union of its
open subsets, and the \emph{boundary} of $\Theta_0$ is the complement
in $\Theta_0$ of its interior.  Note that the boundary and the set of
singularities of a model do not need to be disjoint.

We adopt the regularity conditions of \citet{drton2009}:
$\mathcal{P}_{\Theta}$ is \emph{regular at}
$\theta\in\Theta\subseteq{}\mathbb{R}^k$ if it satisfies the following
conditions: 1) $\theta$ is in the interior of non-empty $\Theta$, 2)
The model $\mathcal{P}_{\Theta}$ is differentiable in quadratic mean
with non-singular Fisher information matrix $\mathcal I\left(\theta\right)$, 3)
For all $\theta_1$, $\theta_2$ in a neighborhood of $\theta$ in
$\Theta$,
$|\log{}p_{\theta_1}\left(x\right)-\log{}p_{\theta_2}\left(x\right)|\leq{}\dot{l}\left(x\right)\left\Vert\theta_1-\theta_2\right\Vert$
for measurable square-integrable function $\dot{l}$, and 4) The
maximum likelihood estimator is a consistent estimator of $\theta$
under $\mathcal{P}_{\Theta}$.

\subsection{Assumptions}

Our models satisfy the following assumptions.

\begin{assump}
\label{ass1}
$\mathcal{P}_{\Theta}$ is regular at generating parameter $\theta_0\in\Theta_0\subset{}\Theta\subseteq{}\mathbb{R}^k$.
\end{assump}

\begin{assump}
\label{ass2}
$\Theta_0$ is Chernoff regular at $\theta_0$ \citep{Chernoff54}, with maximum likelihood estimator $\widehat{\theta}_n$ 
 a consistent estimator of $\theta_0$ under $\mathcal{P}_{\Theta_0}$.
\end{assump}

\begin{assump}
\label{ass3}
$Z_i\in\Theta$, for $i\in\left\{1,2,\ldots{},n\right\}$, are i.i.d
random observations, with finite expected value $\theta_0$ and non-singular covariance $\mathcal{I}\left(\theta_0\right)^{-1}$, where
$\mathcal{I}\left(\theta_0\right)$ is the Fisher information matrix
for a sample of size $1$ and
$\mathcal{I}\left(\theta_0\right)^{\frac{1}{2}}$ a matrix such that
$\mathcal{I}\left(\theta_0\right)=\left(\mathcal{I}\left(\theta_0\right)^{\frac{1}{2}}\right)^T\mathcal{I}\left(\theta_0\right)^{\frac{1}{2}}$.
\end{assump}

We emphasize that there is no assumption that $\mathcal{P}_{\Theta_0}$ is
regular at $\theta_0$, only that $\Theta_0$ is Chernoff regular at
$\theta_0$. This weaker assumption permits $\theta_0$ to be a
singularity and/or boundary of $\Theta_0$. In fact, Chernoff
regularity everywhere in $\Theta_0$ is implied by the assumption that
$\Theta_0$ is a semialgebraic set \citep{drton2009}.

\smallskip

We briefly present and analyze a
simple model, illustrating how the expected AIC 
can fail to converge to its target at a boundary
of a parameter space.

\begin{example}[AIC for a biased coin toss]
\label{simpleexample}

A coin has probability of heads $\theta\in\Theta=\left(0,1\right)$. The submodel $\theta_0\in\Theta_0=\left[\frac{1}{2},1\right)$, in which heads is at least as likely as tails, has a boundary point at $\theta_0=\frac{1}{2}$.  Assume $n$ is sufficiently large to ignore issues near the boundaries $0$ and $1$ of $\Theta$.

The target for the expected AIC has asymptotics determined by the
tangent cone at $\theta_0$ of $\Theta_0$ and the distribution of
$\bar{\mathcal{Z}}_n=\frac{1}{n}\mathcal{Z}_n$, asymptotically
$\mathcal{N}\left(\theta_0,\frac{1}{n}\mathcal{I}\left(\theta_0\right)^{-1}\right)$. For any $\theta_0>\frac{1}{2}$, the tangent cone is a line and the expected AIC converges to its target, with the AIC having bias correction $2k=2$. At $\theta_0=\frac{1}{2}$, the tangent
cone is the half-line $\left[\frac{1}{2},\infty\right)$. Informally,
  asymptotically with probabilities $\frac{1}{2}$, $\Theta_0$
  ``behaves'' either like a point (when $\widehat{\theta}_n=\frac{1}{2}$)
  or like a line (when $\widehat{\theta}_n>\frac{1}{2}$), such that $2k_e=2\left(\frac{1}{2}\cdot{}0+\frac{1}{2}\cdot{}1\right)=1$. Thus the expected AIC is not a consistent estimator of its target, which would require AIC bias correction $1$ instead of $2k=2$. In particular, the asymptotics determined by the tangent cone are discontinuous at the boundary point $\theta_0=\frac{1}{2}$.

Indeed, although the expected AIC converges to its target when
$\widehat{\theta}_n>\frac{1}{2}$, convergence is slow for parameters
near the boundary, when $\theta_0\approx\frac{1}{2}$. 
For such $\theta_0$ it is desirable that the bias correction 
continuously ``interpolate''
between its value at $\theta_0=\frac{1}{2}$ and values at regular
points far away, but in a way dependent on the sample size $n$ that accurately estimates $2k_e$. This is the goal of
developing the generalized AIC.

\end{example}

\subsection{Example models}\label{subsection:phyloModels}

We introduce five  models, all trinomial.  Four are  from
phylogenomics, the field of mathematical biology concerned
with the inference of evolutionary trees from genomic-scale data, and a fifth more general
model. The four phylogenomics models are based on the \emph{multispecies
coalescent} (MSC) model of \emph{incomplete lineage sorting}, which
can result in different evolutionary relationships on different genes.

For three species, $a$, $b$ and $c$, and three orthologous genes $A$, $B$, $C$, descending from a common ancestral gene, there are three
possible rooted gene trees, $A|BC$, $B|AC$ and $C|AB$, describing the gene
triplet's evolutionary history. If an inferred gene tree has topology $B|AC$, for example, then $A$ and $C$ are the two most closely related genes.
Assuming the MSC model,
the most probable gene tree (rooted triple) topology is the one matching the species
tree topology, the overall species history, and the other gene tree triplet
probabilities are equal. 
See Appendix A of \citet{mitchell2019hypothesis} for a brief, but more thorough,
introduction to these trinomial phylogenomic models.

\smallskip

Model $T1$ posits a specific triplet species tree topology. With
$\Theta = \Delta^2$ denoting the open 2-dimensional simplex, its parameter space is the line segment
\begin{align*}
\Theta_0=\left\{\mathbf p\in \Theta \mid p_1\ge p_2=p_3 \right\},
\end{align*}
with a boundary at the centroid $\left(\frac{1}{3},\frac{1}{3},\frac{1}{3}\right)$ of $\Theta$ corresponding to a ``star tree'' evolutionary history. 
Model $T3$ permits any of the three species tree topologies and 
has parameter space \begin{align*}
\Theta_0=\left\{\mathbf p\in \Theta \mid p_i\ge p_j=p_k, \left\{i,j,k\right\}=\left\{1,2,3\right\} \right\},
\end{align*}
with a singularity at the centroid. The \emph{unconstrained} model $U$ has $\Delta^2$ as its parameter space, and might be used to model evolution not on a species tree.  The \emph{polytomy} model, modeling a ``star tree'' where all three species are equally closely related, has the centroid as its parameter space. See Figure \ref{fig:ex12} for the parameter spaces of the first four models. 
The fifth more general model is the \emph{multiple half-lines} model, a specific
generalization of $T1$ and $T3$. All models satisfy Assumptions~\ref{ass1}-\ref{ass3}.

\begin{figure}[!tb]
\begin{subfigure}{0.49\linewidth}
\caption{Model $T1$}
\label{fig:ex12a}
\hspace*{\fill}
\begin{tikzpicture}[xscale=0.98,yscale=0.98]
				\draw[-] (2.31,0) -- (2.31,-2.67);
				\draw[-,dashed] (2.31,0) -- (0,-4);
				\draw[-,dashed] (2.31,0) -- (4.62,-4);
				\draw[-,dashed] (0,-4) -- (4.62,-4);
				\draw (2.31,0) node[circle,draw=black, fill=white,inner sep=1pt,label=above:{$\left(1,0,0\right)$}]{};
				\draw (2.31,-2.67) node[circle,fill,inner sep=1pt,label=below:{$\left( \frac 13,\frac 13,\frac13\right)$}]{};
                \draw (0,-4) node[circle,draw=black, fill=white,inner sep=1pt,label=below:{$\left(0,1,0 \right)$}]{};
                \draw (4.62,-4) node[circle,draw=black, fill=white,inner sep=1pt,label=below:{$\left(0,0,1\right)$}]{};
\end{tikzpicture}
\end{subfigure}
\hfill
\begin{subfigure}{0.49\linewidth}
\caption{Model $T3$}
\label{fig:ex12b}
\begin{tikzpicture}[xscale=0.98,yscale=0.98]
				\draw[-] (0,-4) -- (2.31,-2.67);
                \draw[-] (2.31,0) -- (2.31,-2.67);
                \draw[-] (4.62,-4) -- (2.31,-2.67);
                
				\draw[-,dashed] (2.31,0) -- (0,-4);
				\draw[-,dashed] (2.31,0) -- (4.62,-4);
				\draw[-,dashed] (0,-4) -- (4.62,-4);
				\draw (2.31,0) node[circle,draw=black, fill=white,inner sep=1pt,label=above:{$\left(1,0,0\right)$}]{};
				\draw (2.31,-2.67) node[circle,fill,inner sep=1pt,label=right:{$\left( \frac 13,\frac 13,\frac13\right)$}]{};
                \draw (0,-4) node[circle,draw=black, fill=white,inner sep=1pt,label=below:{$\left(0,1,0 \right)$}]{};
                \draw (4.62,-4) node[circle,draw=black, fill=white,inner sep=1pt,label=below:{$\left(0,0,1\right)$}]{};
				\end{tikzpicture}                
\end{subfigure}
\hspace*{\fill}
\caption{The models $T1$ (a) and $T3$ (b),
  represented by the solid line segment(s) in $\Delta^2$. The centroid
  $\left(\frac{1}{3},\frac{1}{3},\frac{1}{3}\right)$ is a boundary of model $T1$ and a
  singularity of model $T3$,   and also the  polytomy model parameter space. 
  All these models are contained in $\Delta^2$, the unconstrained model parameter space.}
\label{fig:ex12}
\end{figure}
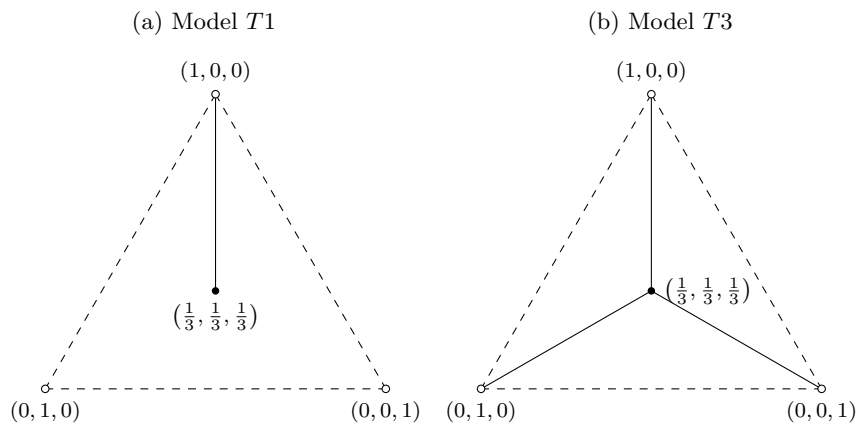

\section{A generalized AIC}
\label{AICg}

Common derivations of the AIC, such as that given by
\citet{cavanaugh1997unifying}, involve Taylor series in $\theta$ at
$\theta_0$ and $\widehat{\theta}_n$.  Since our interest is in
singularities and boundaries, Taylor series fail to exist, and we
follow an alternate framework.

For a sample of size $n$, let
\begin{align}
\label{AICtarget}
\Delta_n=\Delta_n\left(\Theta_0,\theta_0\right)=\mathbb{E}_{0}\left\{\mathbb{E}_0\left\{-2\log{}L\left(\theta|\mathcal{Z}_n\right)\right\}
\big\vert_{\theta=\widehat{\theta}_n}\right\}.
\end{align}
This is the \emph{target} for the AIC and AICg. (See \citet{cavanaugh1997unifying} for the connection between the target $\Delta_n$ and the Kullback-Leibler divergence of the approximating model probability distribution at $\widehat{\theta}_n$ from the generating model probability distribution at $\theta_0$.)

\citet{cavanaugh1997unifying} derives an unbiased estimator of the target for all $\theta_0\in\Theta_0$:
\begin{multline} \label{AICmotivation}
-2\log{}L\left(\widehat{\theta}_n|\mathcal{Z}_n\right) \\
+\mathbb{E}_0\left\{\mathbb{E}_0\left\{-2\log{}L\left(\theta|\mathcal{Z}_n\right)\right\}
\big\vert_{\theta=\widehat{\theta}_n}\right\}-\mathbb{E}_0\left\{-2\log{}L\left(\widehat{\theta}_n|\mathcal{Z}_n\right)\right\}.
\end{multline}
The expected value of Expression~\ref{AICmotivation} yields the target. Its derivation requires no special assumptions
other than well-defined quantities.   Our derivation of the
AICg approximates the last two terms, the \emph{bias correction}, in a
general way.

\begin{definition}[Generalized AIC]
\label{defAICg}
For a model meeting the assumptions of Section~\ref{definitions}, the \emph{generalized AIC} for a sample $\mathcal{Z}_n=\sum_{i=1}^{n}Z_i$ of size $n$ is
\begin{align*}
\text{AICg}=&
\text{AICg}\left({\mathcal Z}_n, \Theta_0, \theta_0\right)\\
=&-2\log{}L\left(\widehat{\theta}_n|\mathcal{Z}_n\right)+2\mathbb{E}_0\left\{n\left(\bar{\mathcal{Z}}_n-\theta_0\right)^T\mathcal{I}\left(\theta_0\right)\left(\widehat{\theta}_n-\theta_0\right)\right\}.
\end{align*}
\end{definition}

\begin{theorem}
\label{generalizedAIC}
Under the assumptions of Section~\ref{definitions}, $\mathbb{E}_0
\left\{\emph{AICg}\right\}$ is asymptotically equivalent to its target
$\Delta_{n}\left(\Theta_0,\theta_{0}\right)$ in the sense that
\begin{align*}
\mathbb{E}_0\left\{\emph{AICg}\right\}-\Delta_{n}\left(\Theta_0,\theta_{0}\right)\to{}0.
\end{align*}
\begin{proof}

We focus first on finding an expression for the target $\Delta_{n}\left(\Theta_0,\theta_{0}\right)$. For an arbitrary $\theta\in\Theta_0$, by the assumptions of Section~\ref{definitions}, with $c$ a constant,
\begin{align}
\label{equation}
-2\log{}L\left(\theta|\mathcal{Z}_n\right)=n\left(\bar{\mathcal{Z}}_n-\theta\right)^T\mathcal{I}\left(\theta_0\right)\left(\bar{\mathcal{Z}}_n-\theta\right)+c+o_p\left(1\right).
\end{align}

Let $A_n$ and $B_n$ be the random variables on the left and right sides of Equation~\ref{equation}, respectively. With subscripts of zero denoting quantities under the generating process with generating parameter $\theta_0 \in \Theta_0$, by Assumption~\ref{ass3}, $\mathbb{E}_0\left\{B_n\right\}<\infty$, and thus also $\mathbb{E}_0\left\{A_n\right\}<\infty$. Thus, $\mathbb{E}_0\left\{A_n\right\}=\mathbb{E}_0\left\{B_n\right\}$ is equivalent to
\begin{align*}
\mathbb{E}_0\left\{-2\log{}L\left(\theta|\mathcal{Z}_n\right)\right\}=\mathbb{E}_0\left\{n\left(\bar{\mathcal{Z}}_n-\theta\right)^T\mathcal{I}\left(\theta_0\right)\left(\bar{\mathcal{Z}}_n-\theta\right)\right\}+c+o\left(1\right).
\end{align*}
Evaluating $\mathbb{E}_0\left\{A_n\right\}$ at $\theta=\theta_0$ and since $\widehat{\theta}_n\xrightarrow{\ p\ }\theta_0$, then by the continuous mapping theorem,
\begin{align*}
\mathbb{E}_0\left\{-2\log{}L\left(\theta|\mathcal{Z}_n\right)\right\}|_{\theta=\theta_0}=\mathbb{E}_0\left\{-2\log{}L\left(\theta|\mathcal{Z}_n\right)\right\}|_{\theta=\widehat{\theta}_n}+o_p\left(1\right).
\end{align*}
Similarly, for $\mathbb{E}_0\left\{B_n\right\}$,
\begin{align*}
&\mathbb{E}_0\left\{n\left(\bar{\mathcal{Z}}_n-\theta\right)^T\mathcal{I}\left(\theta_0\right)\left(\bar{\mathcal{Z}}_n-\theta\right)\right\}|_{\theta=\theta_0}+c+o\left(1\right) \\
=&\mathbb{E}_0\left\{n\left(\bar{\mathcal{Z}}_n-\theta\right)^T\mathcal{I}\left(\theta_0\right)\left(\bar{\mathcal{Z}}_n-\theta\right)\right\}|_{\theta=\widehat{\theta}_n}+c+o_p\left(1\right).
\end{align*}
It follows from $\mathbb{E}_0\left\{A_n\right\}=\mathbb{E}_0\left\{B_n\right\}$ that
\begin{align}
\label{equation2}
&\mathbb{E}_0\left\{-2\log{}L\left(\theta|\mathcal{Z}_n\right)\right\}|_{\theta=\widehat{\theta}_n} \nonumber \\
=&\mathbb{E}_0\left\{n\left(\bar{\mathcal{Z}}_n-\theta\right)^T\mathcal{I}\left(\theta_0\right)\left(\bar{\mathcal{Z}}_n-\theta\right)\right\}|_{\theta=\widehat{\theta}_n}+c+o_p\left(1\right).
\end{align}

Finally, taking the expected value of Equation~\ref{equation2}, which is again finite by Assumption~\ref{ass3}, we obtain the expression for the target:
\begin{align*}
\Delta_n\left(\Theta_0,\theta_0\right)=&\mathbb{E}_0\left\{\mathbb{E}_0\left\{-2\log{}L\left(\theta|\mathcal{Z}_n\right)\right\}|_{\theta=\widehat{\theta}_n}\right\} \\
=&\mathbb{E}_0\left\{\mathbb{E}_0\left\{n\left(\bar{\mathcal{Z}}_n-\theta\right)^T\mathcal{I}\left(\theta_0\right)\left(\bar{\mathcal{Z}}_n-\theta\right)\right\}|_{\theta=\widehat{\theta}_n}\right\}+c+o\left(1\right) \\
=&\mathbb{E}_0\left\{\mathbb{E}_0\left\{n\left(\bar{\mathcal{Z}}_n-\theta_0\right)^T\mathcal{I}\left(\theta_0\right)\left(\bar{\mathcal{Z}}_n-\theta_0\right)\right\}|_{\theta=\widehat{\theta}_n}\right\} \\
&-2\mathbb{E}_0\left\{\mathbb{E}_0\left\{n\left(\bar{\mathcal{Z}}_n-\theta_0\right)^T\mathcal{I}\left(\theta_0\right)\left(\theta-\theta_0\right)\right\}|_{\theta=\widehat{\theta}_n}\right\} \\
&+\mathbb{E}_0\left\{\mathbb{E}_0\left\{n\left(\theta-\theta_0\right)^T\mathcal{I}\left(\theta_0\right)\left(\theta-\theta_0\right)\right\}|_{\theta=\widehat{\theta}_n}\right\}+c+o\left(1\right).
\end{align*}

The quantity simplifies to
\begin{align*}
\Delta_n\left(\Theta_0,\theta_0\right)=\dim\left(\Theta\right)+\mathbb{E}_0\left\{n\left(\widehat{\theta}_n-\theta_0\right)^{T}\mathcal{I}\left(\theta_0\right)\left(\widehat{\theta}_n-\theta_0\right)\right\}+c+o\left(1\right).
\end{align*}

Next, focussing on the first term of the AICg, by similar arguments,
\begin{align*}
\mathbb{E}_{0}\left\{-2\log{}L\left(\widehat{\theta}_n|\mathcal{Z}_n\right)\right\}=&\mathbb{E}_0\left\{n\left(\bar{\mathcal{Z}}_n-\widehat{\theta}_n\right)^T\mathcal{I}\left(\theta_0\right)\left(\bar{\mathcal{Z}}_n-\widehat{\theta}_n\right)\right\}+c+o\left(1\right) \\
=&\dim\left(\Theta\right)-2\mathbb{E}_0\left\{n\left(\bar{\mathcal{Z}}_n-\theta_0\right)^T\mathcal{I}\left(\theta_0\right)\left(\widehat{\theta}_n-\theta_0\right)\right\} \\
&+\mathbb{E}_0\left\{n\left(\widehat{\theta}_n-\theta_0\right)^{T}\mathcal{I}\left(\theta_0\right)\left(\widehat{\theta}_n-\theta_0\right)\right\}+c+o\left(1\right),
\end{align*}
where $c$ is the same constant as in Equation~\ref{equation}.

The result then follows.
\end{proof}

\end{theorem}

Under the assumptions of Section~\ref{definitions}, the second term of the AICg is asymptotically equivalent to the 
bias correction (last two terms) of Expression~\ref{AICmotivation}.

An equivalent expression that may be easier to compute for some models is
\begin{align*}
\text{AICg}=&-2\log{}L\left(\widehat{\theta}_n|\mathcal{Z}_n\right)+\dim\left(\Theta\right)
+\mathbb{E}_0\left\{n\left(\widehat{\theta}_n-\theta_0\right)^T\mathcal{I}\left(\theta_0\right)\left(\widehat{\theta}_n-\theta_0\right)\right\} \\
&-\mathbb{E}_0\left\{n\left(\bar{\mathcal{Z}}_n-\widehat{\theta}_n\right)^T\mathcal{I}\left(\theta_0\right)\left(\bar{\mathcal{Z}}_n-\widehat{\theta}_n\right)\right\}.
\end{align*}

\begin{remark}\label{rmk:linear_transformation}
To derive the AICg for specific models more easily, we apply a linear transformation 
$\Theta\mapsto{}\sqrt{n}\mathcal{I}\left(\theta_0\right)^{\frac{1}{2}}\Theta$ to the bias correction. 
Under this transformation, $\Theta\mapsto{}\mathcal{M}$, $\Theta_0\mapsto{}\mathcal{M}_0$, 
$\theta_0\mapsto{}\mu_0$, $\widehat{\theta}_n\mapsto{}\widehat{\mu}_n$, 
$\bar{\mathcal{Z}}_n\mapsto{}\bar{z}_n$ and 
$\bar{z}_n$ converges in distribution to 
$z\sim\mathcal{N}\left(\mu_0,I\right)$, where $I$ is the identity matrix. After this transformation to the bias correction,
\begin{align}
\label{AICgtransformed}
\text{AICg}=-2\log{}L\left(\widehat{\theta}_n|\mathcal{Z}_n\right)+2\mathbb{E}_0\left\{\left(\bar{z}_n-\mu_0\right)^T\left(\widehat{\mu}_n-\mu_0\right)\right\}.
\end{align}

The value $\left\Vert\mu_0\right\Vert$ has a simple interpretation: It represents the Mahalanobis distance between the generating parameter $\theta_0$ and a singularity or boundary, which is defined to be at the origin. For models with one parameter, $\left\Vert\mu_0\right\Vert$ is the number of standard deviations of $\bar{\mathcal{Z}}_n$ between $\theta_0$ and the singularity or boundary.

Moreover, asymptotically this bias correction is non-negative, which we prove.
\end{remark}

\begin{prop}
	\label{nonnegpenalty} 
	For a model satisfying the assumptions of Section~\ref{definitions}, asymptotically the bias correction of Equation~\ref{AICgtransformed} is non-negative. Specifically,
	\begin{align*}
		2\mathbb{E}_0\left\{\left(\bar{z}_n-\mu_0\right)^T\left(\widehat{\mu}_n-\mu_0\right)\right\}\geq{}0.
	\end{align*}
	\begin{proof}
We prove that
\begin{align*}
\left(\bar z_n-\mu_0\right)^T\left(\widehat{\mu}_n-\mu_0\right)=\left(\bar z_n-\mu_0\right)\cdot{}\left(\widehat{\mu}_n-\mu_0\right)\geq{}0.
\end{align*}
If the dot product is negative, then considering the lengths of the triangle with vertices $\bar{z}_n$, $\mu_0$ and $\widehat{\mu}_n$, we find that $\left\Vert\bar{z}_n-\mu_0\right\Vert<\left\Vert\bar{z}_n-\widehat{\mu}_n\right\Vert$, a contradiction since $\widehat{\mu}_n$ is the maximum likelihood estimate.

	\end{proof}
\end{prop}

\section{Applications of the AICg to example models}
\label{AICapplications}

In this section we derive explicit formulas for the AICg for example models $T1$, $T3$ and the multiple half-lines model. Since the example models consist of half-open line segments
in $\Delta^2$ with interesting geometry at the centroid $\theta_0=\left(\frac{1}{3},\frac{1}{3},\frac{1}{3}\right)$, we replace these line segments with half-lines extending from $\theta_0$ and $\Delta^2$ with $\mathbb{R}^2$. Furthermore, we assume random observations are multivariate normally distributed, as they are asymptotically.

\subsection{Model $T1$}

The linear transformation of Remark~\ref{rmk:linear_transformation} and extension of parameter spaces projects $\Theta=\Delta^2$ onto $\mathbb{R}^2$, with $\mathcal{M}_0=\mathbb{R}^{+}$, $\mu_0=\left(0,\mu_{0,y}\right)$ and $\left(\frac{1}{3},\frac{1}{3},\frac{1}{3}\right)\mapsto{}\left(0,0\right)$. Defining $\phi_0 \in \left(0,1\right]$ via $p_1=1-\frac{2}{3}\phi_0$, then from \citet{mitchell2019hypothesis}, $\mu_{0,y}=\sqrt{2n}\frac{1-\phi_0}{\sqrt{\phi_0\left(3-2\phi_0\right)}}$.

\begin{prop}
\label{propT1}
For model $T1$,
\begin{align*}
\text{AICg}=&-2\log{}L\left(\widehat{\theta}_n|\mathcal{Z}_n\right)+1+\erf\left(\frac{\mu_{0,y}}{\sqrt{2}}\right).
\end{align*}

\begin{proof}

Let $w=\left(x,y\right)^T$ be an arbitrary realization of $\bar{z}_n$ in $\mathcal{M}$, and 
 $m_0= \left(0,m_{0,y}\right)^T$ the point of $\mathcal{M}_0$ closest to $w$ in Euclidean distance.

When $y<0$, $m_0=\left(0,0\right)^T$ and thus 
$\left(w-\mu_0\right)^T\left(m_0-\mu_0\right) = -\mu_{0,y}\left(y-\mu_{0,y}\right)$.

When $y\geq{}0$, $m_0 = \left(0,y\right)^T$ and 
$\left(w-\mu_0\right)^T\left(m_0-\mu_0\right)
=\left(y-\mu_{0,y}\right)^2$.

Then the AICg bias correction is
\begin{align*}
&2\mathbb{E}_0\left\{\left(\bar{z}_n-\mu_0\right)^T\left(\widehat{\mu}_n-\mu_0\right)\right\}\\
=&2\int_{-\infty}^{0}-\mu_{0,y}\left(y-\mu_{0,y}\right)\frac{1}{\sqrt{2\pi}}\exp\left(-\frac{1}{2}\left(y-\mu_{0,y}\right)^2\right)dy \\
&+2\int_{0}^{\infty}\left(y-\mu_{0,y}\right)^2\frac{1}{\sqrt{2\pi}}\exp\left(-\frac{1}{2}\left(y-\mu_{0,y}\right)^2\right)dy \\
=&1+\erf\left(\frac{\mu_{0,y}}{\sqrt{2}}\right).
\end{align*}
\end{proof}

\end{prop}

In Figure~\ref{figureT1} we compare
the performance of the AICg bias correction with known $\mu_{0,y}$, 
the AIC bias correction, and their target.
Since the AICg depends on $\mu_{0,y} = \mu_{0,y}\left(n\right)$, we make these
comparisons for sample sizes $n\in\left\{30,100,1000\right\}$. 
For estimating the target, which is
challenging to determine in closed form, we averaged $10^7$
simulations for each
$\mu_{0,y}\in\left\{0,0.02,0.04,\ldots{},5\right\}$ and fitted a cubic
smoothing spline. We see that the AICg bias correction quickly
converges to its target, accurately interpolating from $1$ (corresponding to an effective
number of parameters $k_e=\frac{1}{2} < 1 = k$) at the boundary point
to $2$ (for effective number of parameters $k_e=1=k$) at points
infinitely far from the boundary.  The interpolating feature of the AICg
addresses in a sample-size dependent
way the practical concern of computing effective numbers of parameters
at generating parameters near the boundary. By way of contrast, 
note that (except for when $n$ is small) the AIC always overestimates the bias correction ($k > k_e$), 
with the magnitude of the error particularly large at points near
the boundary $\mu_{0,y} = 0$, since it assigns a constant bias correction of $2$.
The AIC also converges
slowly to the target, requiring a Mahalanobis distance of $\mu_{0,y}\approx{}2$ 
or more from the boundary for accurate performance.

\begin{figure}[!tb]
\hspace*{\fill}
\begin{subfigure}{.32\linewidth}
  \includegraphics[width=1\linewidth]{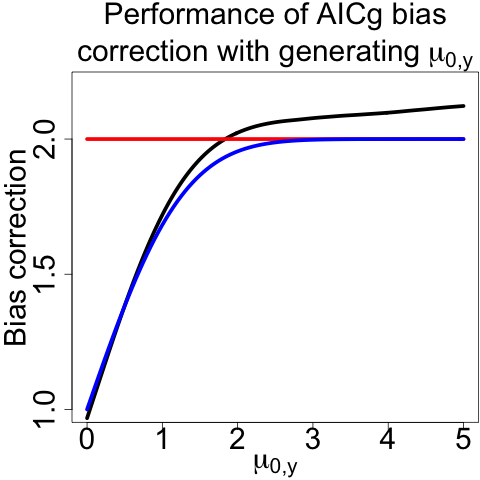}
  \caption{$n=30$}
  \label{fig:NiceImage1}
\end{subfigure}
\hfill
\begin{subfigure}{.32\linewidth}
  \includegraphics[width=1\linewidth]{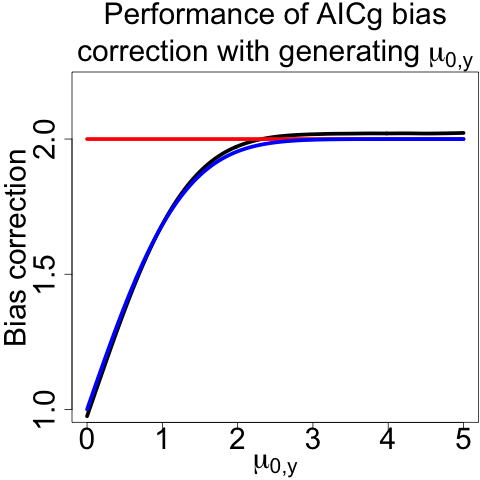}
  \caption{$n=100$}
  \label{fig:NiceImage2}
\end{subfigure}
\hfill
\begin{subfigure}{.32\linewidth}
  \includegraphics[width=1\linewidth]{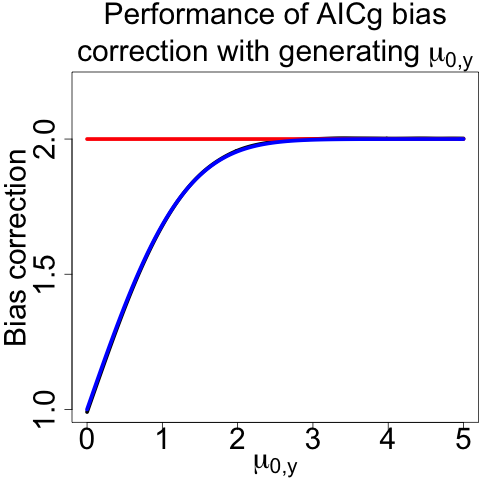}
  \caption{$n=1000$}
  \label{fig:NiceImage3}
\end{subfigure}
\hspace*{\fill} \\
\caption{Performance of AICg bias correction (blue) and AIC bias correction (red) compared to their target estimated by simulation (black) for model $T1$.}\label{figureT1}
\end{figure}

\subsection{Model $T3$}

Again using the transformation of Remark \ref{rmk:linear_transformation} and extension of parameter spaces, model $T3$ maps to three rays in $\mathbb{R}^2$ emanating from $\left(0,0\right)$. They are the non-negative $y$-axis and rays in quadrants $3$ and $4$ forming an angle 
$\alpha_0$ with the negative and positive $x$-axis respectively. If $\theta_0=\left(p_1, p_2, p_3\right)$, then with $p_i$ maximal for some $i \in\left\{1,2,3\right\}$,  $\phi_0\in\left(0,1\right]$ satisfies $p_i=1-\frac{2}{3}\phi_0$. With no loss of generality, we assume $\mu_0=\left(0,\mu_{0,y}\right)$ lies on the non-negative $y$-axis. After transformation, the bias correction depends on $\mu_{0,y}=\mu_{0,y}\left(n\right)$ and $\alpha_0=\arctan\left(\frac{1}{\sqrt{3\left(3-2\phi_0\right)}}\right)$, with $\beta_0=\frac{1}{2}\left(\frac{\pi}{2}-\alpha_0\right)$.

\begin{prop}
\label{propT3}
For model $T3$,
\begin{align*}
\text{AICg}=&-2\log{}L\left(\widehat{\theta}_n|Z_n\right) \\
&+\sqrt{\frac{2}{\pi}}\int_{0}^{\infty}\left(y-\mu_{0,y}\right)^2\exp\left(-\frac{1}{2}\left(y-\mu_{0,y}\right)^2\right)\erf\left(\frac{y\cot\beta_0}{\sqrt{2}}\right)dy \\
&+\frac{2}{\pi}\int_{\frac{-\pi}{2}}^{\beta_0}\int_{0}^{\infty}
g\exp\left(-\frac{1}{2}\left(r^2-2\mu_{0,y}r\sin\phi+\mu_{0,y}^2\right)\right)drd\phi,
\end{align*}
where
\begin{align*}
g=&g\left(r, \phi, \mu_{0,y}, \alpha_0\right) \\
=&r\left(r^2\cos^2\left(\phi+\alpha_0\right)-\mu_{0,y}r\left(\sin\phi-\sin\alpha_{0}\cos\left(\phi+\alpha_0\right)\right)+\mu_{0,y}^2\right).
\end{align*}
\end{prop}

See Appendix~\ref{app1} for the proof.

\smallskip

Although the inner integral of the second term of the bias correction can be evaluated in
closed form, the bias correction as given here is quickly and easily evaluated by numerical integration.

In Figure~\ref{figureT3} we compare the performances of the AICg bias
correction, the AIC bias correction, and 
the approximated target for $\mu_{0,y}$ known. The target is estimated with the same method as for model $T1$.
The AICg bias correction
again quickly converges to its target, accurately interpolating from
$2+\frac{3\sqrt{3}}{2\pi}$ (for effective number of parameters 
$k_e = 1+\frac{3\sqrt{3}}{4\pi}>1=k$) at the singularity $\left(0, 0\right)$ to $2$ (for effective parameters $k_e=1=k$) 
infinitely far from the singularity. Convergence of the AIC is particularly slow at generating parameters near the parameter space singularity. Interestingly, while the AIC underestimated its target for model $T1$, for
this model the AIC overestimates, illustrating that effective number
of parameters can be smaller or larger than $\dim\left(\Theta_0\right)$.

\begin{figure}[!tb]
\hspace*{\fill}
\begin{subfigure}{.32\linewidth}
  \includegraphics[width=1\linewidth]{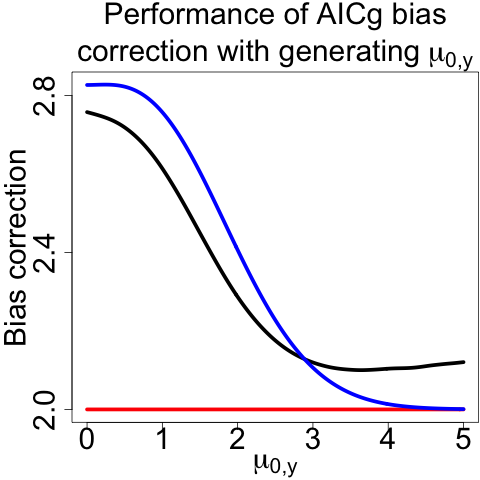}
  \caption{$n=30$}
  \label{fig:NiceImage1}
\end{subfigure}
\hfill
\begin{subfigure}{.32\linewidth}
  \includegraphics[width=1\linewidth]{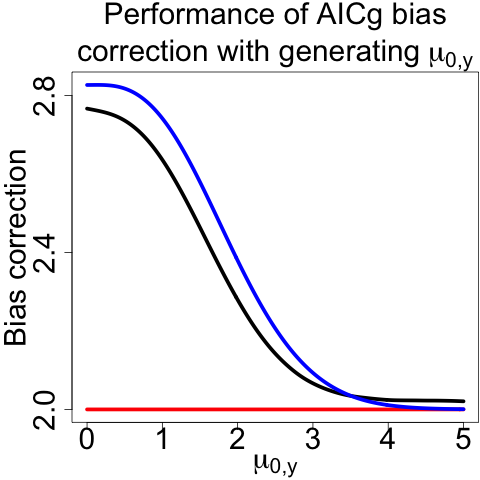}
  \caption{$n=100$}
  \label{fig:NiceImage2}
\end{subfigure}
\hfill
\begin{subfigure}{.32\linewidth}
  \includegraphics[width=1\linewidth]{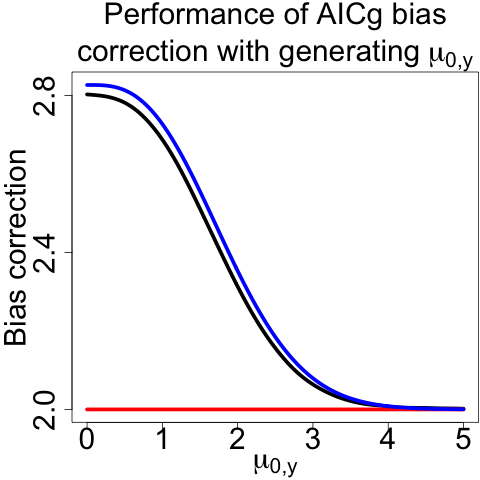}
  \caption{$n=1000$}
  \label{fig:NiceImage3}
\end{subfigure}
\hspace*{\fill} \\
\caption{Performance of AICg bias correction (blue) and AIC bias correction (red) compared to their target
	 estimated by simulation (black) for model $T3$.}\label{figureT3}
\end{figure}

\subsection{Multiple half-lines model}

It is instructive
to consider a generalization of the $T1$ and $T3$ models, composed of
many half-lines in a plane meeting at a singularity.  
This sheds light on how the angles between
these lines influence the bias correction, and thus the notion of
effective number of parameters at their meeting point. We assume $\mathcal{M}_0\subset\mathbb{R}^2$ is the union of
$l\geq{}1$ half-lines emanating from $\left(0,0\right)$.

\begin{prop}
\label{halflines}  
Suppose $\mathcal{M}_0$ for the multiple half-lines model has $l\in{}\mathbb{Z}^{+}$ half-lines, rays at angles $0<\alpha_1<\ldots<\alpha_l=2\pi$ counter-clockwise from the non-negative $x$-axis, with the largest sector that between the non-negative $x$-axis and the ray $\alpha_1$. For each $i\in\left\{1,2,\ldots{},l\right\}$, let $\varphi_i=\alpha_i-\alpha_{i-1}$, with $\alpha_0=0$. Then at $\mu_0=\left(0,0\right)$,
\begin{align*}
\text{AICg}=&\begin{cases}
    -2\log{}L\left(\widehat{\theta}_n|\mathcal{Z}_n\right)+2+\frac{1}{\pi}\sum_{i=1}^{l}\sin\left(\varphi_i\right), & \textrm{if }\varphi_1\in\left(0,\pi\right],
    \\
    -2\log{}L\left(\widehat{\theta}_n|\mathcal{Z}_n\right)+3+\frac{1}{\pi}\left(\sum_{i=2}^{l}\sin\left(\varphi_i\right)-\varphi_1\right), & \textrm{if }\varphi_1\in\left(\pi,2\pi\right].
  \end{cases}
\end{align*}
\end{prop}

For a proof, see Appendix~\ref{app2}. 

\smallskip

The following corollaries, for the special case that all angles $\varphi_i$ are equal, follow easily.

\begin{corollary}
Let $\varphi_i=\frac{2\pi}{l}$ for all $i\in\left\{1,2,\ldots{},l\right\}$. Then
\begin{align*}
\emph{AICg}=&\begin{cases}
    -2\log{}L\left(\widehat{\theta}_n|\mathcal{Z}_n\right)+1, & \textrm{if }l=1, \\
    -2\log{}L\left(\widehat{\theta}_n|\mathcal{Z}_n\right)+2+\frac{l}{\pi}\sin\left(\frac{2\pi}{l}\right), & \textrm{if }l\in\left\{2,3,\ldots\right\}.
  \end{cases}
\end{align*}
\end{corollary}

\begin{corollary}
Let $\varphi_i=\frac{2\pi}{l}$ for all $i\in\left\{1,2,\ldots{},l\right\}$. Then as $l\to\infty$,
\begin{align*}
\emph{AICg}\to&-2\log{}L\left(\widehat{\theta}_n|\mathcal{Z}_n\right)+4.
\end{align*}
\end{corollary}

\smallskip

\begin{remark}
	Note that models $T1$ ($l=1$) and $T3$ ($l=3$) are special cases of the multiple half-lines
	model, when the generating parameter is the boundary/singularity of the parameter space.
	Moreover, as $l \to \infty$, as in the preceding corollary, we obtain the unconstrained model $U$.
	An easy extension to the case that $l = 0$ yields the polytomy model, with a single point
	parameter space.  Since both the polytomy model and $U$ are regular (no singularities or boundaries)
	and linear at all points, the AIC and the AICg coincide for these models.
\end{remark}

\smallskip

In the phylogenomics applications motivating this work, empiricists might use model selection to
choose between a particular rooted triple, say $a|bc$ (in which species $b$ and $c$ are most closely related),
and a star tree showing no pair of the three most closely related.  This requires selecting between model $T1$ and
the polytomy model.  Similarly, AICg model selection might be used to choose between relating three
species with rooted tree ($T3$), or rejecting tree-like evolution if the unconstrained model $U$ is
selected.  In Figure~\ref{modelcomparisons} we show AICg model selection results for a sample
of size $n = 200$, when $\widehat{\mu}_n$ is used to evaluate the bias correction.

\begin{figure}[!tb]
\hspace*{\fill}
\begin{subfigure}{.4\linewidth}
\caption{$T1$ vs polytomy}
  \includegraphics[width=1\linewidth]{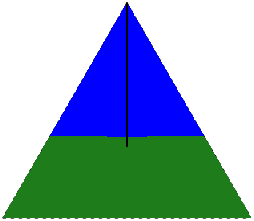}
\end{subfigure}
\hfill
\begin{subfigure}{.4\linewidth}
\caption{$T3$ vs unconstrained}
  \includegraphics[width=1\linewidth]{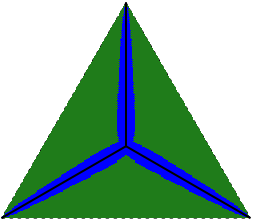}
\end{subfigure}
\hspace*{\fill}
\caption{Model selection outcomes based on the AICg using $\widehat{\mu}_{n}$ and $n=200$. 
	(a) For observations in the blue region, Model $T1$ is selected; for observations
	in green, the polytomy model is selected. 
	(b) Model $T3$ is selected for observations in the blue region; 
	the unconstrained model $U$ is selected in the green region.}
\label{modelcomparisons}
\end{figure}

\section{Using the AICg in practice}
\label{practicalAICg}

As the AICg bias correction may be a function of the unknown generating
parameter, 
practical means for estimating the parameter or directly estimating the bias correction are needed. We restrict ourselves to the transformed basis introduced in Remark~\ref{rmk:linear_transformation} as the generating parameter and bias correction are typically easiest to estimate in this basis.

The simplest solution, using the AICg and estimating $\mu_0$ with $\widehat{\mu}_n$, 
often results in a much more accurate estimate of the target than the AIC does, yet, like the AIC, the AICg with $\widehat{\mu}_n$ is not generally a consistent estimator of the target at singularities and boundaries.
In this section we outline several methods that might be used to 
estimate bias corrections in practice.  In addition to the AIC bias correction and AICg bias correction
computed with the MLE,
these include \emph{least favorable bias corrections},
\emph{uniformly outperforming bias corrections} and
\emph{minimax bias corrections}.
We illustrate how these bias corrections might perform, by comparing and contrasting
them to the AIC bias correction and AICg bias correction with known parameter for example models $T1$ and $T3$. 
We conclude with a brief discussion of  \emph{consistent bias corrections}
and how our ideas might extend to more complicated models.

\subsection{Some practical bias corrections}

\subsubsection{Lower and upper least favorable bias corrections}
\emph{Least favorable} bias corrections are inspired by the least
favorable method of hypothesis testing, dating back to at least
\citet{self1987asymptotic}. 
The lower/upper least favorable bias corrections $c_{llf}/c_{ulf}$ correspond to the 
infimum/supremum of the bias corrections over $\mathcal{M}_0$.

\subsubsection{Uniformly outperforming bias corrections}

\emph{Uniformly outperforming} bias corrections $c_{uo}$ are those that outperform the AIC bias correction; that is, if $c_{uo}$ has bias, it has the same sign as the bias of the AIC bias correction everywhere in $\mathcal{M}_0$, but with smaller magnitude. One such construction is to consider neighborhoods of fixed radii
around each singularity and boundary in $\mathcal{M}_0$, with the
radii chosen to satisfy the uniformly outperforming criterion.  More
specifically, if $\widehat{\mu}_n$ or
$\bar{z}_n$ lies inside a
neighborhood, then the bias correction is taken to be
that at the singularity/boundary within the neighborhood. Otherwise, the bias correction is that of the AIC. If $\widehat{\mu}_n$ or
$\bar{z}_n$ lies inside
multiple neighborhoods, then the neighborhood of the
singularity/boundary closest in Euclidean distance is used, with ties settled
at random.

\subsubsection{Minimax bias corrections}

\emph{Minimax} bias corrections $c_{m}$ are similar to uniformly outperforming
bias corrections, except that the radii  of singularity/boundary neighborhoods
are chosen to minimize the supremum of a risk function over $\mathcal{M}_0$. 
In the comparisons we make below, we choose the risk
function to be the $\mathcal L^2$ norm between the ``true'' bias
correction from known $\mu_0$ and the expected estimated bias
correction from the method described with neighborhoods enclosing singularities and boundaries.

\smallskip

In practice, to determine whether a bias correction is least favorable, uniformly outperforming or minimax, we typically assume $n$ is sufficiently large for the AICg bias correction with known $\mu_0$ to accurately approximate the ``true'' bias correction.

\subsection{Comparison of bias corrections for models $T1$ and $T3$}

The least favorable bias corrections for model $T1$ are $1$, corresponding to
 $\mu_{0,y}=0$, and $2$, as $\mu_{0,y}\to\infty$. A simple uniformly
outperforming bias correction is $c_{uo}=1$ if $\widehat{\mu}_{n,y}=0$,
and $c_{uo}=2$ otherwise, since
$1+\erf\left(\frac{\mu_{0,y}}{\sqrt{2}}\right)<\mathbb{E}_0\left\{c_{uo}\right\}<2$. Our minimax bias correction is $c_{m}=1$
if $\widehat{\mu}_{n,y}\leq{}0.95$, and $c_{m}=2$ otherwise.

For model $T3$, the least favorable bias corrections are $2$
($\mu_{0,y}\to\infty$) and $2+\frac{3\sqrt{3}}{2\pi}$ ($\mu_{0,y}=0$). Finding a strictly uniformly outperforming bias correction for model $T3$ is challenging, however. Instead, we use a bias correction for $T3$ that
\emph{almost} uniformly outperforms the AIC bias correction, which we still refer to as the uniformly outperforming bias correction. The uniformly outperforming bias correction and the AIC bias correction are permitted to have opposite sign biases if the bias of the uniformly outperforming bias correction has small magnitude.

The AICg bias correction depends on $n$ 
in multiple ways, not only through $\mu_{0,y}=\mu_{0,y}\left(n\right)$. 
Thus, we set $n=10^6$ so that the AICg bias correction closely approximates its
target and base all other bias corrections off $n=10^6$. With this procedure, $c_{uo}=2+\frac{3\sqrt{3}}{2\pi}$ if
$\left\Vert\bar{z}_n\right\Vert \le r $ and $c_{uo}=2$ otherwise. Setting the radius $r=1.77$ ensures that if the uniformly outperforming bias correction and AIC bias correction have different sign biases, the magnitude of the bias of the uniformly outperforming bias correction does not exceed $1.02\times{}10^{-14}$. For the minimax bias correction $c_{m}$, we use the same procedure with $r=2.21$.

In Figure~\ref{figureT1a}, we compare the performances of these 
bias corrections for models $T1$ and $T3$ under the assumption that the AICg
bias correction with known $\mu_{0,y}$ closely approximates its
target.   The lower and upper least favorable corrections sandwich
all other bias corrections, and the AIC bias correction is the upper least favorable for $T1$
and the lower least favorable for $T3$.  These are the worst performing bias
corrections, in the sense that they deviate most markedly from the blue
curve, the bias correction of the AICg with known $\mu_0$, which well approximates the target.
Of the data dependent bias corrections, the bias correction of the AICg using $\widehat{\mu}_{n}$, whose expectation was estimated by averaging $10^7$ simulations for each $\mu_{0,y}\in\left\{0,0.02,0.04,\ldots{},5\right\}$, is perhaps the easiest to use for these models and performs well.

\begin{figure}[!tb]
	\hspace*{\fill}
	\begin{subfigure}{.44\linewidth}
		\includegraphics[width=1\linewidth]{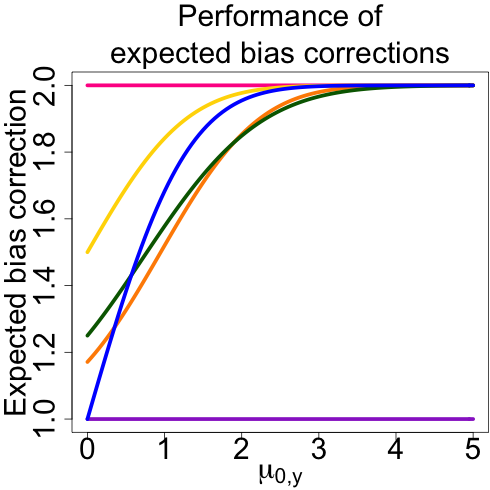}
		\caption{Model $T1$}
		\label{fig:NiceImage4}
	\end{subfigure}
	\hfill
	\begin{subfigure}{.44\linewidth}
		\includegraphics[width=1\linewidth]{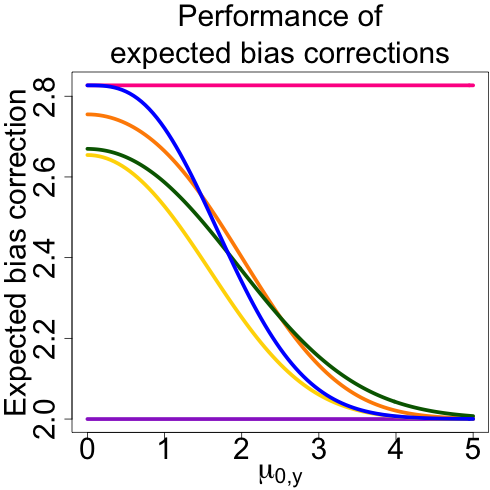}
		\caption{Model $T3$}
		\label{fig:NiceImage4}
	\end{subfigure}
	\hspace*{\fill} \\
	\caption{Performance of expected bias corrections for model $T1$ (left) and model $T3$ (right). Bias corrections are the upper least favorable (also AIC for $T1$) (pink), the uniformly outperforming (yellow), the AICg using $\widehat{\mu}_{n,y}$ (green), the minimax (orange), the AICg using $\mu_{0,y}$ (blue) and the lower least favorable (also AIC for $T3$) (purple).}
	\label{figureT1a}
\end{figure}

\subsection{Consistent bias corrections}

In circumstances that the sample size $n$ can be chosen before an experiment is performed, a consistent
bias correction method may be preferable, with $n$ chosen
sufficiently large to ensure a desired accuracy. Such an estimator
could be obtained using similar procedures to the first method of
\citet{andrews2000inconsistency}. This method is similar to the methods for obtaining uniformly outperforming and minimax bias corrections, except radii of neighborhoods increase with $n$ in such a way that $\widehat{\mu}_n$ or
$\bar{z}_n$ converge
almost surely to be inside/outside a neighborhood if $\mu_0$ is/is not the
singularity or boundary inside the ball.

The radii are chosen to increase with $n$ more slowly than
$\mu_0$ for generic parameters; that is, more slowly than $O\left(\sqrt{n}\right)$. Suppose
$\mu^{\dagger}$ is the closest singularity or boundary to
$\widehat{\mu}_n$ or
$\bar{z}_n$. Then $\mu_0$ 
is estimated by the singularity or boundary if
$\left\Vert{}\widehat{\mu}_n-\mu^{\dagger}\right\Vert\leq{}\sqrt{n} \eta_n$
or
$\left\Vert\bar{z}_n-\mu^{\dagger}\right\Vert\leq{}\sqrt{n} \eta_n$,
where $\sqrt{\frac{\log\log{}n}{n}}<O\left(\eta_n\right)<1$.

The second method of \citet{andrews2000inconsistency} can also be
adapted to consistently estimate the target using a parametric
bootstrap. For this, first obtain a consistent estimate $\widetilde{\mu}_n$
of $\mu_0$, as in the first method above.  Then generate parametric bootstraps
$\widetilde{z}_i^{\ast}\sim\mathcal{N}\left(\widetilde{\mu}_n,I\right)$. For
each bootstrap, estimate the generating parameter
$\widetilde{\mu}_i^{\ast}$ as described in the first method above. Then take the bias correction to be
$\left(\widetilde{z}_i^{\ast}-\widetilde{\mu}_n\right)^T\left(\widetilde{\mu}_i^{\ast}-\widetilde{\mu}_n\right)$. Finally,
average this quantity over all bootstrap replicates to obtain the estimate of the AICg bias
correction.

\subsection{Estimating the bias correction in complicated models}

The AICg bias correction can be difficult to determine for complicated
models. In these circumstances the AIC might be preferable.  This requires a rigorous
justification,  and one  
should establish that $\theta_0$ is likely to be far from any
singularity or boundary, or that the expected AIC is still an accurate estimator of its target, despite the presence of
singularities or boundaries.

For estimating the proximity of $\theta_0$ to any singularity or
boundary, one might use the Mahalanobis distances from $\widehat{\theta}_n$ or
$\bar{\mathcal{Z}}_n$ to singularities/boundaries 
 (equivalently, the Euclidean distance from $\widehat{\mu}_n$ or
$\bar{z}_n$), which can be estimated by parametric or non-parametric bootstrapping. As an alternative, one might conclude that $\theta_0$ is likely far from all singularities or boundaries if the proportion of parametric or non-parametric bootstrap maximum likelihood estimates on the same irreducible component of $\Theta_0$ is high.

\smallskip

When one cannot justify using the AIC and the AICg is difficult to
compute, crude bounds on the AICg bias correction may be more
practical. These bounds may be sufficient for model selection, even if
they are not infima/suprema, particularly when the maximum likelihood term of the AIC/AICg
differs substantially between models, making the bias correction of little consequence. Lower/upper bounds could be
obtained by replacing $\Theta_0$ with affine subspaces/superspaces,
and estimating the bias correction by twice the dimension of the
space. In a similar vein, $\Theta_0$ could be replaced by spaces with
desirable geometric properties, such as half-spaces or orthants, and
the bias correction estimated. From Proposition~\ref{nonnegpenalty}, we conclude that asymptotically one such lower bound for all models is $0$. In short, careful consideration of models 
and their geometries should be undertaken before determining whether to use the AIC.

\section{Discussion}
\label{discussion}

Despite being common in practical applications, the effect that
singularities and boundaries have on the accuracy of the AIC has not
received sufficient attention in the literature to date. Best practice should include attempts to determine whether regularity conditions are satisfied and, if not, the effect of their violation on the accuracy of the AIC. We emphasize the importance of establishing that the AIC consistently estimates its target. If the AIC does not consistently estimate its target or accuracy cannot be quantified, the AICg may be a more appropriate alternative. Estimating the bias correction via bootstrap procedures or crude bounds are both fairly accessible. Crude bounds could be interpreted as being analogous to conservative hypothesis testing procedures; a crude lower/upper bound could be chosen if one wants to preference that model more/less relative to other models.

Indeed as shown here, consideration of the parameter space geometry and its consequences
on model selection performance for generating parameters at or near singularities and boundaries can facilitate more accurate model selection practices.  In contrast to the AIC, 
the expected AICg is a consistent estimator of its target at singularities and boundaries and can converge more quickly than the AIC elsewhere. 

Other model selection procedures,
such as the Bayesian information criterion (BIC) \citep{schwarz1978estimating}, as well as cross-validation, may perform
poorly in the presence of singularities and boundaries. Generalized versions may also be more appropriate.

Methods described here may be more appropriate for models
with high curvature at generating parameters that are not singularities or boundaries. For such geometry
the AIC may converge slowly. Large $n$ may be required for $\Theta_0$ to be
approximately locally linear and for the AIC to be accurate.

\section{Acknowledgements}
The National Institutes of Health [2P20GM103395], an NIGMS
Institutional Development Award (IDeA), and the National Science
Foundation award [DMS 2051760] supported this work in part.

\appendix

\section{Derivation of AICg for Model $T3$}
\label{app1}

\begin{proof}[Proof of Proposition~\ref{propT3}]

Without loss of generality assume that $\mu_0 = \left(0,\mu_{0,y}\right)$ lies on the
vertical half-line of model $T3$. Since $\mathcal M_0$ is
symmetric about the $y$-axis, to determine the bias correction
we integrate
only over the right half-plane $x > 0$, then multiply the result by 2.
To accomplish this integration, 
$2\mathbb{E}_0\left\{\left(\bar{z}_n-\mu_0\right)^T\left(\widehat{\mu}_n-\mu_0\right)\right\}$,
we divide the right half-plane into two regions,  $y>x\tan\beta_0$
and $y<x\tan\beta_0$, since $\bar z_n$ is closer to the vertical model line segment
in the first region, and to the model line segment in Quadrant IV in the second
region. (These are the same angles and model half-lines of Figure~12 of \citet{mitchell2019hypothesis}.)

Using the notation as in the proof of Proposition~\ref{propT1},
when $w=\left(x,y\right)^T$ is in the first region, the closest
model point is $m_0=\left(0,y\right)^T$ and 
$\left(w-\mu_0\right)^T\left(m_0-\mu_0\right) = 
\left(y-\mu_{0,y}\right)^2$.
In this region, 
\begin{align*}
&2\mathbb{E}_0\left\{\left(\bar{z}_n-\mu_0\right)^T\left(\widehat{\mu}_n-\mu_0\right)\right\} \\
=&4\int_{0}^{\infty}\int_{0}^{y\cot\beta_0}\left(y-\mu_{0,y}\right)^2\frac{1}{2\pi}\exp\left(-\frac{1}{2}\left(x^2+\left(y
-\mu_{0,y}\right)^2\right)\right)dxdy \\
=&\sqrt{\frac{2}{\pi}}\int_{0}^{\infty}\left(y-\mu_{0,y}\right)^2\exp\left(-\frac{1}{2}\left(y-\mu_{0,y}\right)^2\right)
\erf\left(\frac{y\cot\beta_0}{\sqrt{2}}\right)dy.
\end{align*}

When $w$ is in the second region, where $y<x\tan\beta_0$, since $\alpha_0$ is acute, 
the closest model point is $m_0=\left(x_0,-x_0\tan\alpha_0\right)$, with $x_0$ to be
determined.
Using polar coordinates for $\bar z_n$, let $r = \sqrt{x^2+y^2}$
and $\varphi \in (- \frac \pi 2, \beta_0) = \arctan \left( \frac {y} {x} \right)$. With this notation, the angle between $m_0$ and
$\bar z_n$ (viewed as vectors) is $\varphi + \alpha_0$.  Thus, 
$x_0 = \left( r \cos (\varphi + \alpha_0) \right) \cos \alpha_{0}$
and $y_ 0 =  \left( r \cos (\varphi + \alpha_0) \right) \sin \alpha_{0}$. Thus,
\begin{align*}
&\left(w-\mu_0\right)^T\left(m_0-\mu_0\right) \\
=&\left[{\begin{array}{cc}
   r\cos\varphi & \sin\varphi-\mu_{0,y} \\
  \end{array} } \right]\left[ {\begin{array}{c}
   r\cos\left(\varphi+\alpha_0\right)\cos\alpha_0 \\
   -r\cos\left(\varphi+\alpha_0\right)\sin\alpha_0-\mu_{0,y} \\
  \end{array}}\right]  \\
=&r^2\cos^2\left(\varphi+\alpha_0\right)-\mu_{0,y}r\left(\sin\varphi-\sin\alpha_{0}\cos\left(\varphi+\alpha_0\right)\right)+\mu_{0,y}^2.
\end{align*}

It follows that, in this region,
\begin{align*}
&2\mathbb{E}_0\left\{\left(\bar{z}_n-\mu_0\right)^T\left(\widehat{\mu}_n-\mu_0\right)\right\} \\
=&\frac{2}{\pi}\int_{-\frac{\pi}{2}}^{\beta_0}\int_{0}^{\infty}r\left(r^2\cos^2\left(\varphi+\alpha_0\right)
-\mu_{0,y}r\left(\sin\varphi-\sin\alpha_{0}\cos\left(\varphi+\alpha_0\right)\right)+\mu_{0,y}^2\right) \\
&\exp\left(-\frac{1}{2}\left(r^2-2\mu_{0,y}r\sin\varphi+\mu_{0,y}^2\right)\right)drd\varphi.
\end{align*}

The result then follows.

\end{proof}

\section{Derivation of the AICg for the multiple half-lines model}
\label{app2}

\begin{proof}[Proof of Proposition~\ref{halflines}] We use the same notation as in the proof of Proposition~\ref{propT1}. We compute 
	$2 \mathbb{E}_0\left\{\left(\bar{z}_n-\mu_0\right)^T\left(\widehat{\mu}_n-\mu_0\right)\right\}
	= 2 \mathbb{E}_0 \left\{  \bar z_n^T\widehat \mu_n  \right\}$
	when $\mu_0=0$, by integrating over sectors of the plane of measure $\varphi_i$.

The value of the AICg depends on the measure of $\varphi_1$, and we suppose first that $\varphi_1\in\left(0,\pi\right]$
	and that $w =  \left(x,y\right) = \left(r \cos \theta, r \sin \theta\right)$ is a realization of $\bar{z}_n$ in the first
	sector $R_1$, where $\theta\in\left[0,\varphi_1\right]$. The closest model point $m_0$ is 
	on the non-negative $x$-axis if $\theta\in\left[0,\frac{\varphi_1}{2}\right]$, and on
	$\alpha_1$ if $\theta\in\left[\frac{\varphi_1}{2},\varphi_1\right]$. When $\theta\in\left[0,\frac{\varphi_1}{2}\right]$,
	$m_0 = \left(x, 0\right) = \left(r \cos \theta, 0\right)$. 
	 When $\theta\in\left[\frac{\varphi_1}{2},\varphi_1\right]$, 
	$m_0 = \left( r \cos \left(\varphi_1 - \theta\right)\cos \varphi_1,  r \cos \left(\varphi_1 - \theta\right)   \sin  \varphi_1 \right)$.
	Thus, over this sector, with area element $dA$, we find
	\begin{align*}
		&2 \iint_{R_1}  \bar z_n^T\widehat \mu_n \frac{1}{2 \pi} \exp\left( - \frac{1}{2}  (x^2+y^2) \right) dA \\
		=&\frac{1}{\pi} \left[\int_0^{\frac{\varphi_1}{2}}  \cos^2\theta d\theta
		          +  \int_{\frac{\varphi_1}{2}}^{\varphi_1} \cos^2\left(\varphi_1 - \theta\right) d\theta\right] \int_0^\infty  r^3 \exp\left( - \frac{1}{2}  r^2 \right) dr  \\
		=&\frac{1}{\pi} \left[ 2 \left(\frac{\varphi_1}{4} + \frac{\sin \varphi_1}{4 } \right)  \right]\cdot{}2 \\
		=&\frac{1}{\pi} \left( \varphi_1 + \sin \varphi_1 \right).
	\end{align*}
Since $\varphi_i\leq{}\varphi_1\leq{}\pi$ for all $i\in\left\{2,3,\ldots,l\right\}$, and the Gaussian
density is symmetric about the origin, we find that 
$$
2 \mathbb{E}_0 \left\{  \bar z_n^T\widehat \mu_n  \right\} = 
\sum_{i=1}^{l} \frac{1}{\pi} \left( \varphi_i  + \sin \varphi_i  \right) = 
2 +  \frac{1}{\pi}\sum_{i=1}^{l} \sin \varphi_i.
$$

For the second case, $\varphi_i\in\left(\pi,2\pi\right]$. 
Then $\varphi_i\in\left(0,\pi\right)$ for all $i\in\left\{2,3,\ldots,l\right\}$.  
If $w = \left(r \cos \theta, r \sin \theta\right)$, then the model point $m_0 = \left(0, 0\right)$ is
closest to $w$ if $\varphi_1\in\left[\frac{\pi}{2},\varphi_1 - \frac{\pi}{2}\right]$.
When $\theta\in\left[0,\frac{\pi}{2}\right]$, then as above 
$$2 \int_0^{\frac{\pi}{2}} \int_0^\infty 
\bar z_n^T\widehat \mu_n \frac{1}{2\pi} \exp \left( - \frac{1}{2} r^2 \right) r dr d\theta =
\frac{1}{\pi} \left( \frac{\pi}{4} + \sin \left( \pi \right) \right) \cdot{}2 = \frac{1}{2} .
$$
Again using that the Gaussian is symmetric, integrating over $\theta \in \left[\varphi_1 - \frac{\pi}{2}, \varphi_1\right]$ 
also yields the value $\frac{1}{2}$.  Thus, when $\varphi_1\in\left(\pi,2\pi\right]$,
$$
	2 \mathbb{E}_0 \left\{  \bar z_n^T\widehat \mu_n  \right\} = \\
\sum_{i=2}^{l} \frac{1}{\pi} \left( \varphi_i  + \sin \varphi_i  \right) + 2\cdot{}\frac{1}{2}
=3+\frac{1}{\pi}\left(\sum_{i=2}^{l}\sin\varphi_i-\varphi_1\right).
$$
\end{proof}

\bibliographystyle{unsrtnat}
\bibliography{bibliography}

\end{document}